\documentclass[reqno,12pt,a4paper]{article}
\usepackage{amscd,amssymb}
\usepackage{verbatim}  
\usepackage{tikz, graphicx}
\usepackage{a4wide}  
\usepackage{amsfonts,latexsym}
\usepackage{amsthm}
\usepackage{amsmath}

\usepackage[abs]{overpic}

\usepackage[numbers,sort&compress]{natbib}
\usepackage{url}
\usepackage[colorlinks=true,allcolors=blue]{hyperref}  % clickable DOI/URL

\usepackage[affil-it]{authblk}
\renewcommand{\baselinestretch}{1.05}
\theoremstyle{plain}
\newtheorem{theorem}{\bf Theorem}[section]
\newtheorem{lemma}[theorem]{\bf Lemma}

\newcommand{\R}{\mathbb{R}}

\newcommand{\Abs}[1]{\left\vert #1 \right\vert}
\newcommand{\abs}[1]{\vert #1 \vert}

\newcommand{\norm}[1]{\left\Vert #1 \right\Vert}

\DeclareMathOperator{\supp}{supp}

\title{A Regularized Shallow Water System} %\vskip 0.1in

\author{
Evgueni Dinvay\footnote{Department of Chemistry,
University of Troms\o , Norway.} \ and Henrik Kalisch\footnote{Department of Mathematics, University of Bergen, Norway.} 
}

\date{}
\begin{document}

%%%%%%%%%%%%%%%%%%%%%%%%%%%%%%%%%%%%%%%%%%%%%%%%%%%%%%%%%%%%%%%%%%%%%%%%%%%%%
\maketitle
%%%%%%%%%%%%%%%%%%%%%%%%%%%%%%%%%%%%%%%%%%%%%%%%%%%%%%%%%%%%%%%%%%%%%%%%%%%%%

%%%%%%%%%%%%%%%%%%%%%%%%%%%%%%%%%%%%%%%%%%%%%%%%%%%%%%%%%%%%%%%%%%%%%%%%%%%%%
\begin{abstract}
The shallow-water system is a standard model for long waves in shallow water. The system is hyperbolic and, for a large class of initial data, solutions develop steep gradients leading to shock formation in finite time. Since such singularities violate the long-wave assumptions underlying the model, their appearance limits the regime of validity of the equations.
In this work, we introduce a regularized shallow-water system in which the nonlinear terms are modified by a bounded operator. This regularization removes the standard derivative-steepening mechanism responsible for shock formation in the classical system while remaining consistent with the long-wave regime. We establish local well-posedness and small-data global well-posedness in Sobolev spaces that exclude singularity formation. In addition, numerical simulations indicate that the system admits solitary-wave solutions.
\end{abstract}
%%%%%%%%%%%%%%%%%%%%%%%%%%%%%%%%%%%%%%%%%%%%%%%%%%%%%%%%%%%%%%%%%%%%%%%%%%%%%

%%%%%%%%%%%%%%%%%%%%%%%%%%%%%%%%%%%%%%%%%%%%%%%%%%%%%%%%%%%%%%%%%%%%%%%%%%%%%
\section{Introduction}
%%%%%%%%%%%%%%%%%%%%%%%%%%%%%%%%%%%%%%%%%%%%%%%%%%%%%%%%%%%%%%%%%%%%%%%%%%%%%
\setcounter{equation}{0}
%%%%%%%%%%%%%%%%%%%%%%%%%%%%%%%%%%%%%%%%%%%%%%%%%%%%%%%%%%%%%%%%%%%%%%%%%%%%%
%
%
%
\noindent
The shallow-water system is one of the most widely known hyperbolic systems of 
partial differential equations \cite{Courant_Friedrichs1999}, 
and it is also one of the most popular systems 
for the modeling of long waves in shallow water \cite{Whitham}.
Numerical approaches to 
discretizing the shallow-water system form the basis of many popular coastal models used in coastal engineering practice and river flow modeling as well as academic research. 
 
As explained for example in \cite{Stoker}, the shallow-water equations are 
usually written in terms of the total flow depth $h(x,t)$ and the average 
horizontal velocity $u(x,t)$ as the unknowns, and it is assumed that 
the pressure is hydrostatic since the flow conditions are changing very slowly.
The most fundamental version of the shallow-water system concerns 
the case of long-crested waves propagating over a flat bed. 
Using the assumption of hydrostatic pressure, 
the equations of mass and momentum conservation take the form
\begin{equation}
\label{SWsystem_1}
\left.
\begin{array}{rcl}
h_t +  (h u )_x & = & 0, \\
(hu)_t + ( hu^2 + g h^2/2)_x          & = & 0,
\end{array}
\right\}
\end{equation}
\noindent
where $g$ denotes the gravitational acceleration.
%\noindent
For smooth solutions and with undisturbed depth $h_0$ and free-surface deflection
from rest denoted by $\eta(x,t)$, the system can be rewritten in the form
\begin{equation}
\label{SWsystem}
\left.
\begin{array}{rcl}
\eta_t + h_0 u_x  + (\eta u )_x & = & 0, \\
u_t + (u^2/2 + g \eta )_x          & = & 0.
\end{array}
\right\}
\end{equation}
As mentioned above, this system applies to long waves in a river or 
in shallow coastal areas, and it can be generalized to a two-dimensional 
version for waves in the coastal region with realistic bathymetry.
Such equations and even more general equations incorporating dispersion
and friction are widely used for numerical simulations of various aspects of
nearshore waves \cite{Bondehagen_Roeber_Kalisch_Buckley2024, Kalisch_Lagona_Roeber2024, Roeber2010}, and the discretizations
used for these equations are usually based on general numerical
methods for hyperbolic systems such as detailed in
%IntermodelComparison
\cite{LeVeque}, for example.

One of the problems with the system shown above is that
even though the derivation explicitly makes use of the assumption that
the waves to be described are long, solutions of this equation
often develop singularities in the form of infinite derivatives.
Indeed, it is shown formally in \cite{Whitham} that waves that
carry a decrease in elevation form derivative singularities. From a modeling point of view, these shocks can be interpreted as bores (see \cite{Ali_Kalisch2010,Favre1935,Whitham}), but from a mathematical point of view they can be a problem.

A rigorous mathematical proof can be given that solutions corresponding to compressive initial configurations develop unbounded gradients,
leading to shock formation (see \cite{Evans}, ch. 11.3).
On the other hand,
it is also known that global smooth solutions may persist
for certain non-compressive or monotone classes of initial data
\cite{Bressan2000, Lax1973},
so shock formation is not universal.

For times up to singularity formation, it can be shown mathematically
that the shallow-water equations arise as a limit point in the convergence 
of the more fundamental free-surface water-wave problem based on the 
full inviscid Euler equations to the shallow-water system
(given some appropriate assumptions on the regularity of the solutions)
\cite{Lannes}.
The problem is that solutions of \eqref{SWsystem} only exist for a short
time for a large class of initial data.
Indeed it can be shown with fairly elementary analysis methods
that solutions always blow up given some mild assumptions (see \cite{Evans}).
Unfortunately, once infinite derivatives and shocks appear, the
very assumption upon which the shallow-water system is based, is invalidated.
This problem is known as the {\it Long wave paradox} in some circles
\cite{LeMehaute,Ursell1953}.

%In the last 150 years a theory of conservation laws
%based on the tracking of such singularities has been developed
%(see Bressan, Courant). There is also an immense literature
%on numerical methods of conservation laws, and the exact description
%of shock waves is often taken as a litmus test for the accuracy
%of the method.

In the present note, it is our purpose to introduce a regularized
system which avoids the problem of singularity formation.
The system we put forward is
\begin{equation}
\label{regularizedSW}
	\left.
	\begin{array}{rcl}
		\eta_t + h_0 u_x  + K ( \eta u )_x
		& = &
		0
		, \\
		u_t + g \eta_x + K ( u^2/2 )_x
		& = &
		0
		,
	\end{array}
	\right\}
\end{equation}
where a regularizing operator can be chosen in
such a way that $K \partial_x$ is bounded in $L^2(\mathbb R)$.
At the same time one can make sure that $K \approx 1$
in the shallow water regime.
Indeed, as an example of such an operator,
the Fourier multiplier
\begin{equation}
\label{K_definition}
	K = \frac{ \tanh h_0D }{ h_0D }
\end{equation}
can be taken.
Here $D = -i\partial_x$,
and so at low frequencies the corresponding symbol $K(\xi) \approx 1$.
With this choice of $K$ we have 
$K \partial_x = i\tanh h_0D / h_0$ that is obviously
a bounded Fourier multiplier operator in $L^2(\mathbb R)$.
Moreover, this regularization naturally arises
in the Hamiltonian long wave approximation
of the full water wave problem \cite{Dinvay_Dutykh_Kalisch}, and
is similar to a Whitham-type system \cite{Aceves_Sanchez_Minzoni_Panayotaros, Moldabayev_Dutykh} although
the regularization is now in the nonlinear terms as opposed to the linear terms.
The total energy of System \eqref{regularizedSW} has the form
\begin{equation}
\label{Hamiltonian}
	\mathcal H(\eta, u)  = \frac 12 \int_\R
	\left(	
		g \eta K^{-1} \eta
		+ h_0 u K^{-1} u
		+ \eta u^2
	\right)
	dx
\end{equation}
where the domain of functional $\mathcal H$
is the Sobolev space $H^{1/2}(\mathbb R) \times H^{1/2}(\mathbb R)$.
Note that System \eqref{regularizedSW} has the following Hamiltonian structure
\begin{equation}
\label{Hamiltonian_structure}
    \eta_t
    =
    - K \partial_x
    \frac{\delta \mathcal H}{\delta u}
    \quad
    \text{ and }
    \quad
    u_t
    =
    - K \partial_x
    \frac{\delta \mathcal H}{\delta \eta}
\end{equation}
with the variational derivatives
\begin{equation}
\label{variational_derivatives}
    \frac{\delta \mathcal H}{\delta \eta}
    =
    g K^{-1} \eta
	+
    \frac 12 u^2
    \quad
    \text{ and }
    \quad
    \frac{\delta \mathcal H}{\delta u}
    =
    h_0 K^{-1} u
    +
    \eta u
    .
\end{equation}
In particular,
we have conservation of $\mathcal H$,
following from the skew-adjointness of $K \partial_x$,
since $K$ is self-adjoint and commutes with $\partial_x$.

We study the Cauchy problem for the semilinear regularization \eqref{regularizedSW}
of the shallow-water system.
The regularization modifies the nonlinear fluxes in such a way
that the derivative-steepening mechanism is suppressed.
As a consequence, we obtain local well-posedness
and global well-posedness for sufficiently small initial data in Sobolev spaces.
The large-data global dynamics is not addressed in this work.

Regularizations of shallow-water type systems that preserve conservative structure
have been proposed in several works.
In particular, a non-dispersive, non-dissipative regularization was introduced in
\cite{Clamond_Dutykh2018non_dispersive_conservative_regularization},
and its shock structure was further analyzed in \cite{Pu_Pego_Dutykh_Clamond2018}.
Extensions to uneven bottom topographies were considered in \cite{Clamond_Dutykh_Mitsotakis2019}.
These works address similar modeling questions from a different perspective
and provide useful context for the present approach.

%%%%%%%%%%%%%%%%%%%%%%%%%%%%%%%%%%%%%%%%%%%%%%%%%%%%%%%%%%%%%%%%%%%%%%%%%%%%%
\section{Local well-posedness}
%%%%%%%%%%%%%%%%%%%%%%%%%%%%%%%%%%%%%%%%%%%%%%%%%%%%%%%%%%%%%%%%%%%%%%%%%%%%%
\setcounter{equation}{0}
%%%%%%%%%%%%%%%%%%%%%%%%%%%%%%%%%%%%%%%%%%%%%%%%%%%%%%%%%%%%%%%%%%%%%%%%%%%%%
%
\noindent
In this section we work in the non-dimensional
settings $h_0 = 1$ and $g = 1$.
Introduce the space $X^s = H^s \times H^s$
equipped with the standard norm.
Denote by $X^s_T$ the space of continuous functions
defined on $[0, T]$ with values in $X^s$,
equipped with the supremum-norm.
%Define matrices
%%
%\[
%	\mathcal W
%	=
%	\frac 1{\sqrt 2}
%	\begin{pmatrix}
%		1 & 1
%		\\
%		W & -W
%	\end{pmatrix}
%	, \quad
%	\mathcal W^{-1}
%	=
%	\frac 1{\sqrt 2}
%	\begin{pmatrix}
%		1 & W^{-1}
%		\\
%		1 & -W^{-1}
%	\end{pmatrix}
%	.
%\]
%
%
Regard the unitary group
\[
	\mathcal S(t) =
	\frac 1{\sqrt 2}
	\begin{pmatrix}
		1 & 1
		\\
		1 & -1
	\end{pmatrix}
	\begin{pmatrix}
		e^{-itD} & 0
		\\
		0 & e^{itD}
	\end{pmatrix}
	\frac 1{\sqrt 2}
	\begin{pmatrix}
		1 & 1
		\\
		1 & -1
	\end{pmatrix}
\]
that is a composition of isometric operators in $X^s$.
Note that for any $s, t \in \mathbb R$,
$v \in X^s$ holds
\(
	\lVert \mathcal S(t)v \rVert_{X^s}
	=
	\lVert v \rVert_{X^s}
	,
\)
and consequently
\(
	\lVert \mathcal S(t)v \rVert_{X^s_T}
	=
	\lVert v \rVert_{X^s_T}
\)
for any $T>0$,
due to the fact that symbols of
eigenvalues of $\mathcal S(t)$ have absolute value
equal to one.
For any fixed $v_0 = (\eta_0, u_0)^T \in X^s$
function $\mathcal S(t)v_0$ solves the linear
initial-value problem associated with
\eqref{regularizedSW}.
Regard a mapping $\mathcal A : X^s_T \to X^s_T$
defined by
\begin{equation}
\label{contraction_mapping}
	\mathcal A(\eta, u) =
	\mathcal A(\eta, u; v_0)(t) = \mathcal S(t)v_0
	- \int_0^t \mathcal S(t - t') K \partial_x
	\begin{pmatrix}
		\eta u
		\\
		u^2 / 2
	\end{pmatrix}
	(t')dt'
	.
\end{equation}
Then the Cauchy problem for System \eqref{regularizedSW}
with the initial data $v_0$ may be rewritten
equivalently
as an equation in $X^s_T$ of the form
\begin{equation}
\label{u_is_Au}
	 v = \mathcal A(v; v_0)
\end{equation}
where $v = (\eta, u)^T \in X^s_T$.
Below the latter integral equation is solved locally in time
by making use of Picard iterations.
\begin{lemma}[local well-posedness]
\label{local_well_posedness_lemma}
	Let $s > 1/2$, $v_0 = (\eta_0, u_0)^T \in X^s$
	and $T = ( 7C_s \lVert v_0 \rVert_{X^s} )^{-1}$
	with some constant $C_s > 0$ depending only on $s$.
	Then there exists a unique solution
	$v = (\eta, u)^T \in X^s_T$ of Problem \eqref{u_is_Au}.

	Moreover, for any $R > 0$ there exists $T = T(R) > 0$
	such that the flow map associated with Equation \eqref{u_is_Au}
	is a real analytic mapping of the open ball
	$B_R(0) \subset X^s$ to $X^s_T$.
\end{lemma}
\begin{proof}
The proof idea is to show that the restriction of $\mathcal A$
on some closed ball $B_M$ centered at $\mathcal S(t)v_0$
is a contraction mapping.
The Sobolev space $H^s$ is an algebra for $s > 1/2$, and so
there is a positive constant $C_s > 0$
such that 
\[
	\lVert ( \eta u, u^2/2 ) \rVert _{X^s}
	\leqslant
	C_s \lVert ( \eta, u ) \rVert _{X^s}^2
\]
and
\[
	\lVert ( \eta_1 u_1 - \eta_2 u_2,
	u_1^2/2 - u_2^2/2 ) \rVert _{X^s}
	\leqslant
	C_s \lVert ( \eta_1 - \eta_2, u_1 - u_2 ) \rVert _{X^s}
	(
		\lVert ( \eta_1, u_1 ) \rVert _{X^s}
		+
		\lVert ( \eta_2, u_2 ) \rVert _{X^s}
	)
	.
\]

Thus for any $T, M > 0$ and $v, v_1, v_2 \in B_M \subset X^s_T$ hold
\[
	\lVert \mathcal A(v) - \mathcal S(t)v_0 \rVert _{X^s_T}
	\leqslant
	\int_0^T \lVert ( \eta u, u^2/2 ) \rVert _{X^s}
%	\leqslant
%	T \lVert ( \eta u, u^2/2 ) \rVert _{X^s_T}
	\leqslant
	C_sT \lVert v \rVert _{X^s_T}^2
	,
\]
\[
	\lVert \mathcal A(v_1) - \mathcal A(v_2) \rVert _{X^s_T}
	\leqslant
	C_sT \lVert v_1 - v_2 \rVert _{X^s_T}
	(
		\lVert v_1 \rVert _{X^s_T}
		+
		\lVert v_2 \rVert _{X^s_T}
	)
	,
\]
and so taking $M = 2\lVert v_0 \rVert _{X^s}$
and $T$ as in the formulation of the lemma
we conclude that $\mathcal A$ is a contraction in
the closed ball $B_M$.
The first statement of the lemma follows from the
contraction mapping principle.

We turn our attention to the smoothness of the flow map.
Let $R > 0$, $T = ( 7C_s R )^{-1}$ and
$B = B_R(0)$ be an open ball in $X^s$.
Define $\Lambda : B \times X^s_T \to X^s_T$ as
\[
	\Lambda(v_0, v) = v - \mathcal A(v; v_0)
\]
that is obviously a smooth map.
Its Fr\'echet derivative with respect to the second variable
is defined by
\[
	d_v\Lambda(v_0, v)h =
	h + \int_0^t \mathcal S(t - t') K \partial_x
	\begin{pmatrix}
		u & \eta
		\\
		0 & u
	\end{pmatrix}
	h(t')dt'
\]
where $v = (\eta, u)^T$ and $h \in X^s_T$.
If $v_1 \in X^s_T$ is the solution of Problem \eqref{u_is_Au}
corresponding to the initial data $v_0 \in B$ then
$\Lambda(v_0, v_1) = 0$.
Moreover, it satisfies the following estimate
\[
	\lVert v_1(t) \rVert_{X^s}
	\leqslant
	\lVert v_0 \rVert_{X^s} +
	C_s \int_0^t \lVert v_1(t') \rVert_{X^s}^2 dt'
\]
and so
\[
	\int_0^t \lVert v_1(t') \rVert_{X^s}^2 dt'
	\leqslant
	\frac{ t \lVert v_0 \rVert_{X^s}^2 }
	{1 - C_st \lVert v_0 \rVert_{X^s}}
\]
for any $t$.
The latter is used to estimate operator
$ I - d_v\Lambda(v_0, v_1) $ as follows
\begin{multline*}
	\lVert h - d_v\Lambda(v_0, v_1)h \rVert
	\leqslant
	C_s \sup _{t \in [0, T]} \int_0^t
	\lVert v_1(t') \rVert_{X^s}
	\lVert h(t') \rVert_{X^s} dt'
	\\
	\leqslant
	C_s \sup _{t \in [0, T]}
	\left(
		t \int_0^t \lVert v_1(t') \rVert _{X^s}^2 dt'
	\right) ^{1/2}
	\lVert h \rVert_{X^s_T} 
	\leqslant
	\frac{ C_sT \lVert v_0 \rVert_{X^s} }
	{ \sqrt{ 1 - C_sT \lVert v_0 \rVert_{X^s} } }
	\lVert h \rVert_{X^s_T} 
	\leqslant
	\frac 1
	{ \sqrt{ 42 } }
	\lVert h \rVert_{X^s_T} 
\end{multline*}
which is true for any $h \in X^s_T$.
As a result operator $d_v\Lambda(v_0, v_1)$ is invertible
and so the second assertion of the lemma follows from
the Implicit Function Theorem.
\end{proof}

\begin{lemma}
\label{classical_solution_lemma}
    %Let $s > 3/2$.
    %Then
    The solution $v(t)$ described in Lemma \ref{local_well_posedness_lemma}
    is classical,
    that is,
    \[
        v \in C^1 \left( (0, T); H^{s - 1}(\R) \times H^{s - 1}(\R) \right)
    \]
    and it satisfies \eqref{regularizedSW}.
\end{lemma}
\begin{proof}
We denote by
\begin{equation}
\label{A_and_f_definition}
    A
    =
    \begin{pmatrix}
        0
        &
        -\partial_x
        \\
        -\partial_x
        &
        0
    \end{pmatrix}
    \quad
    \text{ and }
    \quad
    f(v)
    =
    f(\eta, u)
    =
	-
    K \partial_x
	\begin{pmatrix}
		\eta u
		\\
		u^2 / 2
	\end{pmatrix}
    .
\end{equation}
It is enough to demonstrate that
\begin{equation*}
    \norm{
        \frac{v(t + h) - v(t)}h
        -
        Av(t)
        -
        f(v(t))
    }_{X^{s - 1}}
    \to 0
\end{equation*}
as $h \to 0$ for any $t \in (0, T)$.
Taking into account that $v$ satisfies the integral equation \eqref{u_is_Au},
one can regroup the functions staying under the norm as
\begin{multline*}
    \frac{v(t + h) - v(t)}h
    -
    Av(t)
    -
    f(v(t))
    =
    \underbrace{
        \left(
            \frac{ \mathcal S(h) - I }h - A
        \right)
        v(t)
    }_{
        = \mathfrak g_1(h)
    }
    +
    \underbrace{
        ( \mathcal S(h) - I )
        f(v(t))
    }_{
        = \mathfrak g_2(h)
    }
    \\
    +
    \mathcal S(h)
    \left(
        \underbrace{
            \frac 1h \int_t^{t + h}
            \mathcal S(t - r) ( f(v(r)) - f(v(t)) )
            dr
        }_{
            = \mathfrak g_3(h)
        }
        +
        \underbrace{
            \frac 1h \int_t^{t + h}
            \mathcal ( S(t - r) - I ) f(v(t))
            dr
        }_{
            = \mathfrak g_4(h)
        }
    \right)
    .
\end{multline*}
The first term $\mathfrak g_1(h)$ has the squared norm
\[
    \norm{ \mathfrak g_1(h) }_{X^{s- 1}}^2
    =
    \frac 12 \sum_{m = \pm 1}
    \int
    \Abs{ \widehat{\eta}(\xi) + m \widehat{u}(\xi) }^2
    \frac{ \Abs{ e^{ - m i h\xi} - 1 + m i h\xi }^2 }{ |h|^2 }
    \left( 1 + \xi^2 \right)^{s - 1}
    d \xi
\]
which tends to zero by the dominated convergence theorem,
due to
\[
    \frac{ \Abs{ e^{ \pm i h\xi} - 1 \mp i h\xi } }{ |h| }
    \to 0
    \quad
    \text{ and }
    \quad
    \frac{ \Abs{ e^{ \pm i h\xi} - 1 \mp i h\xi }^2 }{ |h|^2 }
    \left( 1 + \xi^2 \right)^{s - 1}
%    \leqslant
%    4 \xi^2
%    \left( 1 + \xi^2 \right)^{s - 1}
    \leqslant
    4
    \left( 1 + \xi^2 \right)^s
    ,
\]
since
\(
    \eta, u \in H^s(\R)
    .
\)
The second term $\mathfrak g_2(h)$ tends to zero
by the continuity of the semigroup $\mathcal S$ in $X^{s - 1}$,
which also follows from the dominated convergence theorem by the similar argument.
Next,
\[
    \norm{ \mathfrak g_3(h) }_{X^{s- 1}}
    \leqslant
	2 C_s \lVert v \rVert _{X^s_T}
    \Abs{
        \frac 1h
        \int_t^{t + h} \lVert v(r) - v(t) \rVert _{X^s} dr
    }
    \to 0
\]
and
\[
    \norm{ \mathfrak g_4(h) }_{X^{s- 1}}
    \leqslant
    \Abs{
        \frac 1h
        \int_t^{t + h} \lVert \mathcal ( S(t - r) - I ) f(v(t)) \rVert _{X^{s - 1}} dr
    }
    \to 0
\]
by continuity of the integrands.
\end{proof}

We turn our attention to the conservation of $\mathcal H(v(t))$
with the Hamiltonian $\mathcal H(v)$ given in \eqref{Hamiltonian}.
First, we notice that is it well defined for $v \in X^{1/2}$.
Indeed, on the one hand we have an equivalence of the energy norm
\begin{equation}
\label{energy_norm}
    \norm{v}_{\mathcal H}
    =
    \frac 1{\sqrt 2}
    \norm{K^{-1/2} v}_{L^2}
\end{equation}
and $X^{1/2}$-norm,
with
\(
    \norm{v}_{\mathcal H}^2
\)
constituting the quadratic integral in \eqref{Hamiltonian}.
On the other hand the cubic integral
in \eqref{Hamiltonian} is bounded as
\[
    \Abs{
        \int \eta u^2
    }
    \leqslant
    \norm{\eta}_{L^2}
    \norm{u}_{L^4}^2
    \leqslant
    C
    \norm{\eta}_{L^2}
    \norm{u}_{H^{1/4}}^2
    \leqslant
    C
    \norm{v}_{X^{1/2}}^3
\]
by the Sobolev embedding $H^{1/4}(\R) \hookrightarrow L^4(\R)$.
As a matter of fact,
$\mathcal H$ is continuously differentiable on $X^{1/2}$.
Thus taking into account the Hamiltonian structure \eqref{Hamiltonian_structure}
and the result of Lemma \ref{classical_solution_lemma},
one can easily verify the energy conservation by the chain rule for $s \geqslant 3/2$.
An extension of the energy conservation to $s \in (1/2, 3/2)$ is achieved
by the regularization argument,
the most technical part of the proof of the next lemma.

\begin{lemma}
The energy functional $\mathcal H$ is continuously differentiable on $X^{1/2}$.
Moreover,
\[
    \mathcal H(v(t))
    =
    \mathcal H(v_0)
    , \quad
    t \in [0, T],
\]
where $v, v_0$ and $T$ are described in Lemma \ref{local_well_posedness_lemma}.
\end{lemma}

\begin{proof}
Linearisation of $\mathcal H$ around $v \in X^{1/2}$
gives
\[
    d \mathcal H(v) \delta v
    =
	\int_\R
	\left(
		\eta K^{-1} \delta \eta
		+
        u K^{-1} \delta u
		+
        \frac 12 \delta \eta u^2
        +
        \eta u \delta u
	\right)
	dx
    , \quad
    \delta v = ( \delta \eta, \delta u )^T \in X^{1/2}
    ,
\]
which defines a bounded linear operator $d \mathcal H(v)$,
due to
\begin{multline*}
    \Abs{d \mathcal H(v) \delta v}
    \leqslant
    \norm{
		K^{-1/2} \eta
    }_{L^2}
    \norm{
        K^{-1/2} \delta \eta
    }_{L^2}
    +
    \norm{
        K^{-1/2} u
    }_{L^2}
    \norm{
        K^{-1/2} \delta u
    }_{L^2}
    \\
    +
    \frac 12
    \norm{
        \delta \eta
    }_{L^2}
    \norm{
        u
    }_{L^4}^2
    +
    \norm{
        \eta
    }_{L^4}
    \norm{
        u
    }_{L^4}
    \norm{
        \delta u
    }_{L^2}
    \lesssim
    \left(
        1 + \norm{v}_{X^{1/2}}^2
    \right)
    \norm{\delta v}_{X^{1/2}}
    ,
\end{multline*}
and so $\mathcal H$ is differentiable.
Its gradient with respect to the energy inner product, inducing the norm \eqref{energy_norm},
has the form
\[
    \nabla \mathcal H(v)
    =
    2 K
    \frac{ \delta \mathcal H }{ \delta v }
    =
    \left(
        2 K
        \frac{ \delta \mathcal H }{ \delta \eta }
        ,
        2 K
        \frac{ \delta \mathcal H }{ \delta u }
    \right)^T
    =
    \begin{pmatrix}
        2 \eta + K \left( u^2 \right)
        \\
        2 u + 2 K \left( \eta u \right)
    \end{pmatrix}
    ,
\]
where the variational derivatives are defined in \eqref{variational_derivatives}.
Continuity of the derivative $d \mathcal H$ is equivalent to
the continuity of the  gradient $\nabla \mathcal H$.
The latter follows from the estimate
\begin{multline*}
    \norm{ \nabla \mathcal H(v_1) - \nabla \mathcal H(v_2) }_{\mathcal H}
    \leqslant
    2
    \norm{ v_1 - v_2 }_{\mathcal H}
    +
    \frac 1{\sqrt 2}
    \norm{
        \begin{pmatrix}
            u_1^2 - u_2^2
            \\
            2 \eta_1 u_1 - 2 \eta_2 u_2
        \end{pmatrix}
    }_{L^2}
    \\
    \lesssim
    \left(
        1
        +
        \norm{ v_1 }_{\mathcal H}
        +
        \norm{ v_2 }_{\mathcal H}
    \right)
    \norm{ v_1 - v_2 }_{\mathcal H}
    ,
\end{multline*}
where we have used the boundedness of $K^{1/2}$ in $L^2(\R)$,
H\"older's inequality
\(
    \norm{fg}_{L^2}
    \leqslant
    \norm{f}_{L^4}
    \norm{g}_{L^4}
\)
and the Sobolev embedding $H^{1/4}(\R) \hookrightarrow L^4(\R)$.
This completes the proof of smoothness of $\mathcal H$.

In order to prove the energy conservation,
we introduce a regularized version  of System \eqref{regularizedSW} as follows.
For $\mu \geqslant 1$ let $P_\mu$
be the Fourier multiplier with an even smooth symbol $\theta_\mu(\xi) = \theta(\xi/\mu)$,
that is, for $f \in \mathcal S'(\R)$,
\[
    \widehat{P_\mu f}(\xi) = \theta\left( \frac{\xi}{\mu} \right) \widehat f(\xi)
    .
\]
We assume that $\supp\theta \subset [-2,2]$,
hence $\widehat{P_\mu f}$ is supported in $[-2\mu,2\mu]$, so $P_\mu f(x)$ is smooth.
We also assume that $\abs{\theta} \leqslant 1$.
Then $P_\mu$ is bounded in $H^s(\R)$ for any $s$, with operator norm $\leqslant 1$. 
Finally, we assume that $\theta = 1$ on $[-1, 1]$,
so that $P_\mu$ converges strongly to the identity operator
in $H^s(\R)$ as $\mu \to \infty$.

Now consider the following frequency-truncated version
\begin{equation}
\label{frequency_truncatedSW}
    \partial_t v = A P_{\mu} v + f(v)
    , \quad
    v(0) = v_0
\end{equation}
of the Cauchy problem for System \eqref{regularizedSW}
with the initial data $v_0$,
where $A$ and $f$ are defined in \eqref{A_and_f_definition}.
As above we can introduce the unitary group
\(
    \mathcal S_{\mu}(t) = e^{A P_{\mu} t}
\)
and notice that a solution of \eqref{frequency_truncatedSW}
satisfies the integral equation
\begin{equation}
\label{frequency_truncated_duhamel}
	v(t)
    =
    \mathcal S_{\mu}(t - t_0) v(t_0)
	+
    \int_{t_0}^t \mathcal S_{\mu}(t - r) f(v(r))
	dr
    , \quad
    0 \leqslant t_0 \leqslant t
	.
\end{equation}
On $X^{1/2}$ we define
\begin{equation*}
%\label{Hamiltonian}
	\mathcal H_{\mu}(v)
    =
	\mathcal H_{\mu}(\eta, u)
    =
    \frac 12 \int_\R
	\left(	
		\eta K^{-1} P_{\mu} \eta
		+
        u K^{-1} P_{\mu} u
		+
        \eta u^2
	\right)
	dx
\end{equation*}
that is obviously smooth.
It is straightforward to repeat the above arguments and conclude
that there exists a unique solution
\[
    v = v_{\mu}
    \in
    C   \left( [0, T]; X^s \right)
    \cap
    C^1 \left( (0, T); X^s \right)
\]
of \eqref{frequency_truncatedSW} for every given $\mu$.
Importantly,
here $s$ and $T$ are exactly as in the formulation of Lemma \ref{local_well_posedness_lemma}.
In particular, $T$ is independent of $\mu$.
Clearly,
\(
	\mathcal H_{\mu}(v(t))
\)
is smooth with respect to $t \in (0, T)$ and the chain rule
\[
    \frac d{dt}
	\mathcal H_{\mu}(v(t))
    =
    d \mathcal H_{\mu}(v(t))
    \partial_t v(t)
    =
	\int_\R
	\left(
		\eta K^{-1} P_{\mu} \partial_t \eta
		+
        u K^{-1} P_{\mu} \partial_t u
		+
        \frac 12 u^2 \partial_t \eta
        +
        \eta u \partial_t u
	\right)
	dx
\]
applies.
It is easy to see that \eqref{frequency_truncatedSW} has
the Hamiltonian structure \eqref{Hamiltonian_structure}
with the new functional $\mathcal H_{\mu}$ in place of $\mathcal H$
and the new variational derivatives
\begin{equation*}
    \frac{\delta \mathcal H_{\mu}}{\delta \eta}
    =
    K^{-1} P_{\mu} \eta
	+
    \frac 12 u^2
    \quad
    \text{ and }
    \quad
    \frac{\delta \mathcal H_{\mu}}{\delta u}
    =
    K^{-1} P_{\mu} u
    +
    \eta u
    .
\end{equation*}
Therefore,
\begin{multline*}
    \frac d{dt}
	\mathcal H_{\mu}(v(t))
    =
	\int_\R
	\left(
        \frac{\delta \mathcal H_{\mu}}{\delta \eta}
        \partial_t \eta
		+
        \frac{\delta \mathcal H_{\mu}}{\delta u}
        \partial_t u
	\right)
	dx
    =
    -
	\int_\R
	\left(
        \frac{\delta \mathcal H_{\mu}}{\delta \eta}
        K \partial_x
        \frac{\delta \mathcal H_{\mu}}{\delta u}
		+
        \frac{\delta \mathcal H_{\mu}}{\delta u}
        K \partial_x
        \frac{\delta \mathcal H_{\mu}}{\delta \eta}
	\right)
	dx
    \\
    =
    -
	\int_\R
    \partial_x
	\left(
        \frac{\delta \mathcal H_{\mu}}{\delta \eta}
        K
        \frac{\delta \mathcal H_{\mu}}{\delta u}
	\right)
	dx
    =
    0
    ,
\end{multline*}
and so $\mathcal H_{\mu}(v(t))$ is constant.
By continuity,
\(
    \mathcal H_{\mu}(v(t)) = \mathcal H_{\mu}(v_0)
\)
for all $t \in [0, T]$,
where
$v = v_{\mu}$ is the solution of the regularized Cauchy problem \eqref{frequency_truncatedSW}.

Now let $v$ be the solution of the original Cauchy problem described
by Lemma \ref{local_well_posedness_lemma}
and $v_{\mu}$ be the solution of \eqref{frequency_truncatedSW}.
We will show that $v_{\mu} \to v$ in $L^2(0, T; X^s)$ as $\mu \to \infty$.
The idea is to split the whole interval $[0, T]$ into small sub-intervals
of size $h > 0$ and to prove this convergence repeatedly on every sub-interval
moving from left to right.
Both $v$ and $v_{\mu}$ satisfy the integral equation \eqref{frequency_truncated_duhamel}
with $\mathcal S$ and $\mathcal S_{\mu}$ in place of the semigroup, respectively.
Hence
\begin{equation*}
	v_{\mu}(t)
    -
	v(t)
    =
    \mathcal S_{\mu}(t - t_0) v_{\mu}(t_0)
    -
    \mathcal S(t - t_0) v(t_0)
	+
    \int_{t_0}^t
    \left(
        \mathcal S_{\mu}(t - r) f(v_{\mu}(r))
        -
        \mathcal S(t - r) f(v(r))
	\right)
    dr
\end{equation*}
for $0 \leqslant t_0 \leqslant t \leqslant t_0 + h \leqslant T$
with $h$ to be chosen below independently of $\mu$.
We estimate this difference as
\begin{multline*}
	\norm{
        v_{\mu}(t)
        -
    	v(t)
    }_{X^s}
    \leqslant
	\norm{
        v_{\mu}(t_0)
        -
    	v(t_0)
    }_{X^s}
    +
	\norm{
        \left(
            \mathcal S_{\mu}(t - t_0)
            -
            \mathcal S(t - t_0)
        \right)
        v(t_0)
    }_{X^s}
	\\
    +
    \int_{t_0}^t
    \norm{
        f(v_{\mu}(r))
        -
        f(v(r))
	}_{X^s}
    dr
    +
    \int_{t_0}^t
    \norm{
        \left(
            \mathcal S_{\mu}(t - r)
            -
            \mathcal S(t - r)
    	\right)
         f(v(r))
	}_{X^s}
    dr
    .
\end{multline*}
Squaring this inequality and then integrating the result from $t_0$ to $t_0 + h$, one obtains
\begin{multline*}
    \frac 1{16}
	\norm{
        v_{\mu}
        -
    	v
    }_{ L^2(t_0, t_0 + h; X^s) }^2
    \leqslant
    h
	\norm{
        v_{\mu}(t_0)
        -
    	v(t_0)
    }_{X^s}^2
    \\
    +
    \int_{t_0}^{t_0 + h}
	\norm{
        \left(
            \mathcal S_{\mu}(t - t_0)
            -
            \mathcal S(t - t_0)
        \right)
        v(t_0)
    }_{X^s}^2
    dt
    +
    h^2
    \int_{t_0}^{t_0 + h}
    \norm{
        f(v_{\mu}(r))
        -
        f(v(r))
	}_{X^s}^2
    dr
	\\
    +
    h
    \int_{t_0}^{t_0 + h}
    \int_{t_0}^{t_0 + h}
    \norm{
        \left(
            \mathcal S_{\mu}(t - r)
            -
            \mathcal S(t - r)
    	\right)
         f(v(r))
	}_{X^s}^2
    dr dt
    .
\end{multline*}
From the proof of Lemma \ref{local_well_posedness_lemma} one can easily deduce
\(
	\norm{
        v_{\mu}
    }_{X_T^s}
    ,
	\norm{
    	v
    }_{X_T^s}
    \leqslant
	3
    \norm{
        v_0
    }_{X^s}
    ,
\)
and so
\[
    h^2
    \int_{t_0}^{t_0 + h}
    \norm{
        f(v_{\mu}(r))
        -
        f(v(r))
	}_{X^s}^2
    dr
    \leqslant
    C h^2
	\norm{
        v_{\mu}
        -
    	v
    }_{ L^2(t_0, t_0 + h; X^s) }^2
    ,
\]
where $C$ is independent of $\mu$.
This implies
\begin{multline*}
    \frac 1{20}
	\norm{
        v_{\mu}
        -
    	v
    }_{ L^2(t_0, t_0 + h; X^s) }^2
    \leqslant
    h
	\norm{
        v_{\mu}(t_0)
        -
    	v(t_0)
    }_{X^s}^2
    \\
    +
    \int_{t_0}^{t_0 + h}
    \int_{\mathbb R}
    \Abs{
        \mathcal F
        \left[
            \left(
                \mathcal S_{\mu}(t - t_0)
                -
                \mathcal S(t - t_0)
            \right)
            v(t_0)
        \right]
        (\xi)
	}^2
    \left( 1 + \xi^2 \right)^s
    d \xi
    dt
	\\
    +
    h
    \int_{t_0}^{t_0 + h}
    \int_{t_0}^{t_0 + h}
    \int_{\mathbb R}
    \Abs{
        \mathcal F
        \left[
            \left(
                \mathcal S_{\mu}(t - r)
                -
                \mathcal S(t - r)
        	\right)
             f(v(r))
        \right]
        (\xi)
	}^2
    \left( 1 + \xi^2 \right)^s
    d \xi
    dr dt
    .
\end{multline*}
These two integrals tend to zero as $\mu \to \infty$ by the dominated convergence theorem,
since
\[
    e^{ \pm i (t - r) \xi \theta_{\mu}(\xi) }
    -
    e^{ \pm i (t - r) \xi }
    \to
    0
    \text{ as }
    \mu \to \infty
    \text{ for every }
    t, r, \xi
\]
and the integrands are bounded by
\[
    2
    \Abs{
        \mathcal F
        \left[
            v(t_0)
        \right]
        (\xi)
	}^2
    \left( 1 + \xi^2 \right)^s
    \quad
    \text{ and }
    \quad
    2
    \Abs{
        \mathcal F
        \left[
             f(v(r))
        \right]
        (\xi)
	}^2
    \left( 1 + \xi^2 \right)^s
    ,
\]
respectively.
The first term
\(
	\norm{
        v_{\mu}(t_0)
        -
    	v(t_0)
    }_{X^s}
    \to
    0
\)
by choosing $t_0$ properly at least for a subsequence $\mu = \mu_k \to \infty$.
Indeed, for the first subinterval $t_0 = 0$ and this term is zero.
Therefore,
\(
    v_{\mu} \to v
\)
in
\(
    L^2(0, h; X^s)
\)
and so a subsequence
\(
    \norm{v_{\mu_k} - v}_{X^s}
    \to
    0
\)
almost everywhere on $[0, h]$.
Next candidate $t_0$ for left boundary of subinterval is taken in $(h/2 ,h]$,
so that
\(
    \norm{v_{\mu_k}(t_0) - v(t_0)}_{X^s}
    \to
    0
    .
\)
Thus by induction we obtain
\(
    \norm{v_{\mu_k} - v}_{X^s}
    \to
    0
\)
in $L^2(0, T)$
and almost everywhere on $[0, T]$.

Finally, we can prove the energy conservation
by analyzing the difference
\begin{equation}
\label{energy_conservation_difference}
    \mathcal H(v(t))
    -
    \mathcal H(v_0)
    =
    \mathcal H(v(t))
    -
    \mathcal H_{\mu}(v_{\mu}(t))
    +
    \mathcal H_{\mu}(v_0)
    -
    \mathcal H(v_0)
\end{equation}
for $t \in [0, T]$ and with $\mu = \mu_k$ described above.
Clearly,
\[
    \Abs{
        \mathcal H_{\mu}(v_0)
        -
        \mathcal H(v_0)
    }
    \leqslant
    \norm{v_0}_{\mathcal H}
    \norm{( P_{\mu} - I )v_0}_{\mathcal H}
    \to
    0
    .
\]
The first difference in \eqref{energy_conservation_difference}
is bounded by
\[
    \Abs{
        \mathcal H(v(t))
        -
        \mathcal H_{\mu}(v_{\mu}(t))
    }
    \leqslant
    \Abs{
        \mathcal H(v(t))
        -
        \mathcal H_{\mu}(v(t))
    }
    +
    \Abs{
        \mathcal H_{\mu}(v(t))
        -
        \mathcal H_{\mu}(v_{\mu}(t))
    }
    ,
\]
where the first term is estimated in the same way
\[
    \Abs{
        \mathcal H(v(t))
        -
        \mathcal H_{\mu}(v(t))
    }
    \lesssim
    \norm{( P_{\mu} - I )v(t)}_{\mathcal H}
    \to
    0
    ,
\]
while the last term is estimated by the mean-value theorem as follows.
The modified energy $\mathcal H_{\mu}$ is a continuously differentiable functional,
which one can check by repeating the arguments given
for analysis of $\mathcal H$ at the beginning of the proof.
Moreover,
\[
    \norm{d \mathcal H_{\mu}(w)}
    \leqslant
    C
    \left(
        1 + \norm{w}_{X^{1/2}}^2
    \right)
    \leqslant
    C
    \left(
        1 + \norm{w}_{X^s}^2
    \right)
\]
with the constant $C$ independent of $\mu$.
Substituting $w = \lambda v(t) + (1 - \lambda) v_{\mu}(t)$
with $0 < \lambda < 1$
and recalling that both $v(t)$ and $v_{\mu}(t)$ are bounded independently of $\mu$,
one obtains
\[
    \Abs{
        \mathcal H_{\mu}(v(t))
        -
        \mathcal H_{\mu}(v_{\mu}(t))
    }
    \leqslant
    C
    \left(
        1 + 9 \norm{v_0}_{X^s}^2
    \right)
    \norm{
        v(t)
        -
        v_{\mu}(t)
    }_{X^s}
    .
\]
The latter converges for almost every $t \in [0, T]$.
Thus, passing to the limit $\mu \to \infty$ in \eqref{energy_conservation_difference},
we obtain
\(
    \mathcal H(v(t))
    -
    \mathcal H(v_0)
    =
    0
\)
for almost every $t \in [0, T]$.
The proof is concluded
by continuity of
\(
    \mathcal H(v(t))
    .
\)
\end{proof}

The next step is to extend the local result of these lemmas.

%\textbf{ \textcolor{red}{to be continued ...} }

%%%%%%%%%%%%%%%%%%%%%%%%%%%%%%%%%%%%%%%%%%%%%%%%%%%%%%%%%%%%%%%%%%%%%%%%%%%%%
\section{Global well-posedness}
%%%%%%%%%%%%%%%%%%%%%%%%%%%%%%%%%%%%%%%%%%%%%%%%%%%%%%%%%%%%%%%%%%%%%%%%%%%%%
\setcounter{equation}{0}
%%%%%%%%%%%%%%%%%%%%%%%%%%%%%%%%%%%%%%%%%%%%%%%%%%%%%%%%%%%%%%%%%%%%%%%%%%%%%
%
\noindent
In this section we prove the following theorem.

\begin{theorem}
[Global existence]
\label{mainthm}
	Let $s > 1/2$.
	There exists $\delta>0$ such that if
	$$
        \eta_0, u_0 \in H^s(\R)
        \quad
        \text{ and }
        \quad
		\| \eta_0\|_{H^{1/2}(\R)}+ \| u_0\|_{H^{1/2}(\R)}
		\leqslant \delta
	$$
	then the Cauchy problem for \eqref{regularizedSW}
	with this initial data
	has a global-in-time solution 
	$$
		(\eta, u) \in C\left( \R; H^s(\R) \times H^s(\R)\right).
	$$
	Moreover,
	the solution is unique and
	depends real analytically on the initial data $(\eta_0, u_0) $
	on any finite time interval $[0, T]$.
\end{theorem}

Firstly, note that
the continuity argument provides us with boundedness
of $X^{1/2}$-norm of a solution $v = (\eta, u)$.
It is more convenient to work with the energy norm \eqref{energy_norm} instead.
Indeed,
there exists $C > 0$ such that
\[
	\lVert v \rVert ^2_{\mathcal H}
	( 1 - C \lVert v \rVert _{\mathcal H} )
	\leqslant
	\mathcal H(v)
	\leqslant
	\lVert v \rVert ^2_{\mathcal H}
	( 1 + C \lVert v \rVert_{\mathcal H} )
	.
\]
Take $\delta = (4C)^{-1}$
and regard a solution with $v_0 = v(0)$
having $\lVert v_0 \rVert _{\mathcal H} \leqslant \delta$.
By continuity
$\lVert v \rVert _{\mathcal H} \leqslant 2 \delta$
on some time interval $[0, T_{\delta}]$,
where $T_{\delta}$ is the greatest lower bound of admissible $t$
with $\lVert v(t) \rVert _{\mathcal H} > 2\delta$,
and so
\[
	\lVert v \rVert _{\mathcal H}
	\leqslant
	\sqrt{2 \mathcal H(v)} = \sqrt{2 \mathcal H(v_0)}
	\leqslant
	\sqrt{ 2 (1 + C \delta) } \delta < 2 \delta
\]
which means that
the continuous function
$\lVert v(t) \rVert _{\mathcal H}$ cannot reach the level
$2\delta$ as time evolves.

\begin{lemma}
[Persistence of regularity]
\label{persistence_lemma}
	Suppose $s > 1/2$ and a pair
	$\eta(t) \in H^s$, $u(t) \in H^s$
	solves System \eqref{regularizedSW}.
	Then the estimate
	\[
		\frac{d}{dt} \lVert (\eta, u) \rVert _{X^s}
		\lesssim
		\left(
			\lVert \eta \rVert _{L^{\infty}}
			+ \lVert u \rVert _{L^{\infty}}
		\right)
		\lVert (\eta, u) \rVert _{X^s}
	\]
    holds true.
\end{lemma}

\begin{proof}
Calculate the derivative of the elevation norm
\begin{equation*}
	\frac 12 \frac{d}{dt} \lVert \eta \rVert _{H^s}^2
	=
	- \int \left( J^s \eta \right)
	J^s \partial_x u
	- i \int \left( J^s \eta \right)
	J^s \tanh D(\eta u)
\end{equation*}
and the derivative of the velocity norm
\begin{equation*}
	\frac 12 \frac{d}{dt} \lVert u \rVert _{H^s}^2
	=
	- \int \left( J^s u \right)
	J^s \partial_x \eta
	- \frac i2 \int \left( J^s u \right) J^s \tanh D u^2
	.
\end{equation*}
After summation of these two lines,
the first parts on the right hand side will cancel each other.
To estimate the second parts notice
\[
	\lVert J^s \tanh D(\eta u) \rVert _{L^2}
	\lesssim
	\lVert \eta \rVert _{H^s}
	\lVert u \rVert _{L^{\infty}}
	+
	\lVert \eta \rVert _{L^{\infty}}
	\lVert u \rVert _{H^s}
\]
and
\begin{equation*}
	\left \lVert J^s \tanh D u^2 \right \rVert _{L^2}
	\lesssim
	\lVert u \rVert _{L^{\infty}}
	\left \lVert u \right \rVert _{H^s}
	.
\end{equation*}
We conclude the proof by applying Hölder’s inequality.
\end{proof}

We will also need the following
Gronwall type inequality.
For the proof, see \cite{Dinvay2019}, for example.

\begin{lemma}
[Gronwall inequality]
	Let $y$ be an absolutely continuous non-negative
	function defined on some interval $[0, T]$.
	Suppose that almost everywhere
	\[
		y' \leqslant y \log y
		.
	\]
	Then there exists $C > 0$ independent of $T$
	such that
	\[
		y(t) \leqslant \exp \left( Ce^t \right)
		.
	\]
\end{lemma}

Our variables $\eta$, $u$ are bounded in
$H^{1/2}$-norm at least for small initial data and
so we are able to use
the following limiting case of the Sobolev embedding theorem.

\begin{lemma}
[Brezis-Gallouet-Wainger inequality]
\label{Brezis_lemma}
	Suppose $f \in H^s(\mathbb R^n)$ with $s > n / 2$.
	Then
	\begin{equation}
	\label{Brezis_inequality}
		\lVert f \rVert_{L^{\infty}}
		\leqslant
		C_{s, n}
		\left(
			1 + \lVert f \rVert_{H^{n/2}}
			\sqrt{ \log( 2 + \lVert f \rVert_{H^s} ) }
		\right)
		.
	\end{equation}
\end{lemma}

Inequality \eqref{Brezis_inequality} was first put forward
and proved for a domain in $\mathbb R^n$ with $n = 2$
in the work by Brezis, Gallouet
\cite{Brezis_Gallouet}.
It was extended to the other Sobolev spaces in
\cite{Brezis_Wainger}.
These two inequalities allow us to close
the regularity persistence argument.

\begin{proof}[Proof of Theorem \ref{mainthm}]
Suppose $s > 1/2$ and
$v(t) = ( \eta(t), u(t) ) \in X^s$ solves
System \eqref{regularizedSW} on some time interval.
Let its initial data $v(0)$ be small
with respect to $X^{1/2}$-norm.
Then $v(t)$ stays bounded in $X^{1/2}$
by a constant independent of the time interval.
Hence from the Brezis-Gallouet-Wainger limiting embedding
\eqref{Brezis_inequality} one deduces
\[
	\lVert \eta(t) \rVert _{L^{\infty}}
	+
	\lVert u(t) \rVert _{L^{\infty}}
	\lesssim 1 +
	\log \left( 2 + \lVert v(t) \rVert _{X^s} \right)
    .
\]
Here we used the elementary estimate $\sqrt{\log z} \leqslant 1 + \log z$
valid for $z \geqslant 1$.
Now applying Lemma \ref{persistence_lemma} we obtain
\[
	\frac{d}{dt} \lVert v \rVert _{X^s}
	\lesssim
	\left(
		1 +
		\log \left( 2 + \lVert v \rVert _{X^s} \right)
	\right)
	\lVert v \rVert _{X^s}
	.
\]
Finally, applying the Gronwall inequality we come to the estimate
\[
	\lVert v \rVert _{X^s}
	\leqslant
	\exp
	\left(
		C e^t
	\right)
	,
\]
where constant $C$ depends only on $s$,
$\lVert v(0) \rVert _{X^{1/2}}$ and
$\lVert v(0) \rVert _{X^s}$.
This estimate allows one to reapply the
local result and extend the solution on any time interval.
\end{proof}

%%%%%%%%%%%%%%%%%%%%%%%%%%%%%%%%%%%%%%%%%%%%%%%%%%%%%%%%%%%%%%%%%%%%%%%%%%%%%
\section{Numerics}
%%%%%%%%%%%%%%%%%%%%%%%%%%%%%%%%%%%%%%%%%%%%%%%%%%%%%%%%%%%%%%%%%%%%%%%%%%%%%
\setcounter{equation}{0}
%%%%%%%%%%%%%%%%%%%%%%%%%%%%%%%%%%%%%%%%%%%%%%%%%%%%%%%%%%%%%%%%%%%%%%%%%%%%%
%
\noindent
We continue to work in the non-dimensional settings
$h_0 = 1$ and $g = 1$.
In this section we deal with numerical simulations
of the shallow water systems \eqref{SWsystem} and
\eqref{regularizedSW}.
To make the comparison more objective,
we also solve numerically
a fully dispersive analogue of \eqref{regularizedSW}
studied thoroughly in \cite{Dinvay, Dinvay_Dutykh_Kalisch}
and having the form
\begin{equation}
\label{regularizedHP}
	\left.
	\begin{array}{rcl}
		\eta_t + h_0 u_x  + K ( \eta u )_x
		& = &
		0
		, \\
		u_t + g K \eta_x + K ( u^2/2 )_x
		& = &
		0
		.
	\end{array}
	\right\}
\end{equation}
This model demonstrates a good agreement with the
full water wave simulations and so can be a reference for comparison
in the present paper.
For time evolution simulations we apply a spectral split-step scheme
described in \cite{Dinvay_Dutykh_Kalisch}.
The scheme has proved to be very efficient both
with respect to speed and accuracy of calculations.
It works for Systems \eqref{regularizedHP},
\eqref{regularizedSW} and before shock development
for the standard shallow water system \eqref{SWsystem}.
The computational domain is $[- 60, 60]$,
the number of Fourier modes is set to $1024$ and
the time step is $0.05$.
No de-aliasing strategy was found to be necessary.

In Figure \ref{GaussEvolution} evolution of a Gaussian profile is depicted
for different models.
\begin{figure}[t!]
	\centering
%	\subfigure
	{
		\includegraphics[width=0.9\textwidth]
		{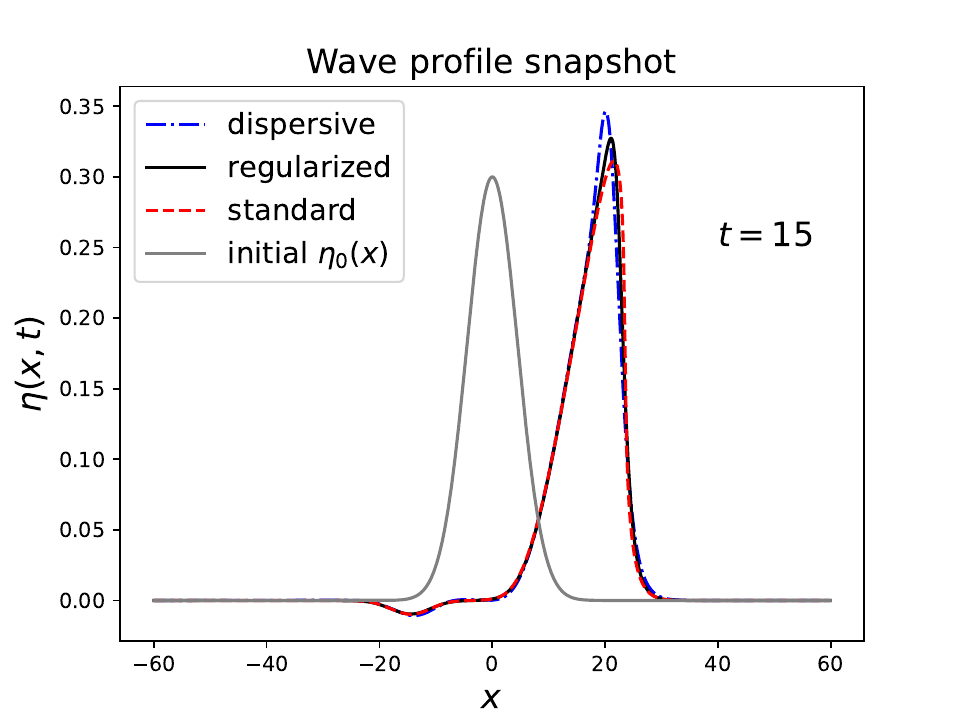}
	}
	\caption
	{
		The thin gray curve represents the initial surface elevation $\eta_0(x)$.
		All solutions are given at time $t=15$
		just before the shallow water model
		\eqref{SWsystem}, corresponding to the red dashed line, is breaking.
		The black solid curve corresponds to the regularized
		shallow water system \eqref{regularizedSW}
		and the blue dash-dot line -- to the fully dispersive model
		\eqref{regularizedHP}.
	}
\label{GaussEvolution}
\end{figure}
\begin{figure}[t!]
	\centering
%	\subfigure
	{
		\includegraphics[width=0.9\textwidth]
		{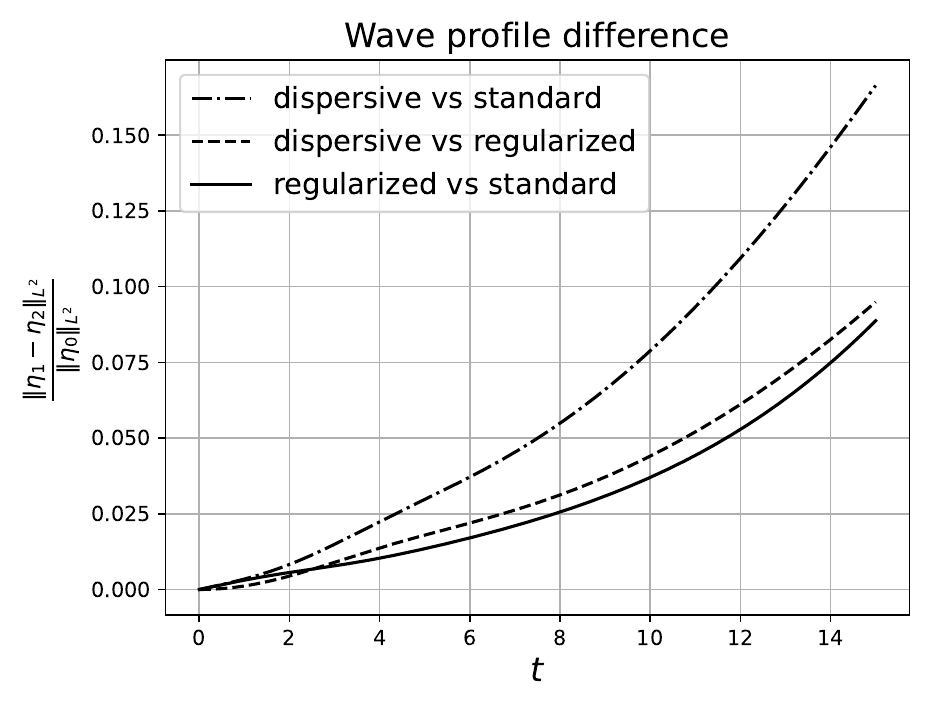}
	}
	\caption
	{
		The difference between solutions of
        the standard shallow water model \eqref{SWsystem},
        the regularized shallow water model \eqref{regularizedSW}
        and the fully dispersive model \eqref{regularizedHP}.
	}
\label{DifferenceEvolution}
\end{figure}
The initial surface elevation is taken the same for all
the models and it has the Gaussian form
\[
	\eta_0(x) = 0.3 \exp \left( -x^2 / 40 \right)
	.
\]
It is depicted by the gray line in Figure \ref{GaussEvolution}.
We set
\(
	u_0 = \eta_0
    ,
\)
which guarantees that most part of the wave profile will travel to the right,
by linearization argument. 
The details
can be found in \cite{Dinvay_Dutykh_Kalisch},
where it is discussed how
imposing the initial velocity
\(
	u_0 = \sqrt{ K } \eta_0
\)
one can create quasi-right-moving waves for \eqref{regularizedHP}.
Since $K \approx 1$ at low frequencies, the practical difference
between the different models under consideration is small.
As a result one can clearly see that most part of the initial wave
moved to the right in Figure \ref{GaussEvolution}.
The conservation of the discrete Hamiltonian was monitored:
$\mathcal H(v(t))$ monotonically decreased from $\mathcal H(v(0))$,
that is close to 0.8 for each model,
by $6.4 \cdot 10^{-13}$ for \eqref{regularizedHP},
by $6.7 \cdot 10^{-13}$ for \eqref{regularizedSW}
and 
by $1.9 \cdot 10^{-10}$ for \eqref{SWsystem}.
It is evident from the figure that all three models
steepen, however, only the solution of \eqref{SWsystem} develops a shock.
Moreover, the difference between solutions of \eqref{SWsystem}
and \eqref{regularizedSW} seems negligible,
as seen in Figure \ref{DifferenceEvolution},
so one may prefer to use the regularization \eqref{regularizedSW}
instead of a dissipative type scheme \cite{Toro} treating
the shock-developing system \eqref{SWsystem}.
We remark that the positivity of the total depth $h = 1 + \eta$
is preserved through all our calculations.
\begin{figure}[t!]
	\centering
%	\subfigure
	{
		\includegraphics[width=0.9\textwidth]
		{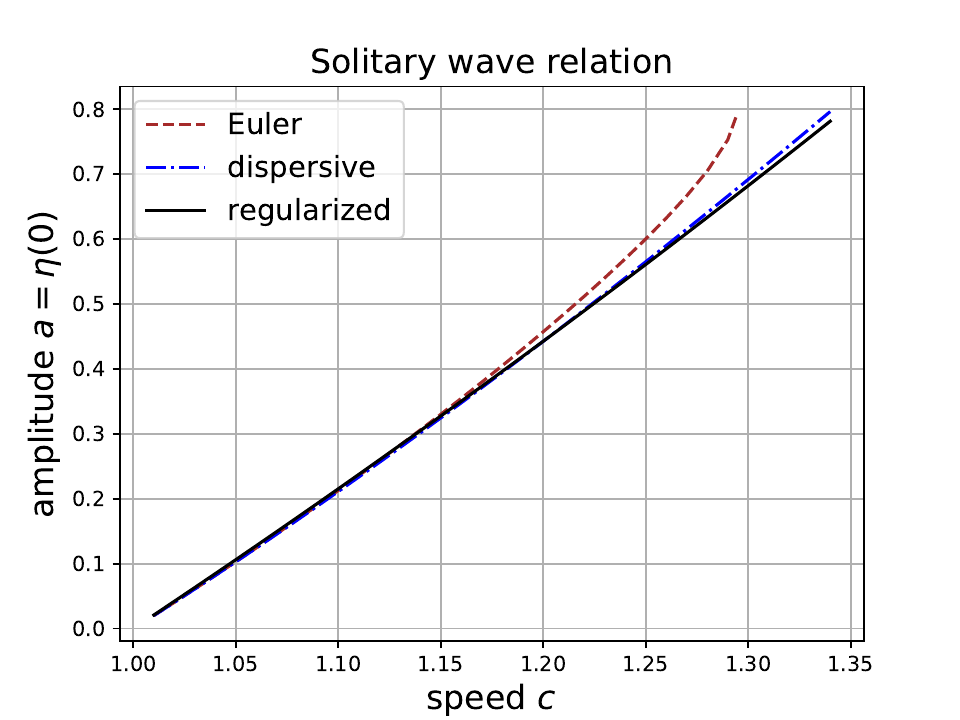}
	}
	\caption
	{
		Amplitude versus speed relation $a(c)$
		of solitary waves:
		the Euler system -- brown dashed line,
		the regularized shallow water system \eqref{regularizedSW}
		-- black solid line,
		the fully dispersive model
		\eqref{regularizedHP} -- blue dash-dot line.
	}
\label{AmplitudeSpeed}
\end{figure}
\begin{figure}[t!]
	\centering
%	\subfigure
	{
		\includegraphics[width=0.9\textwidth]
		{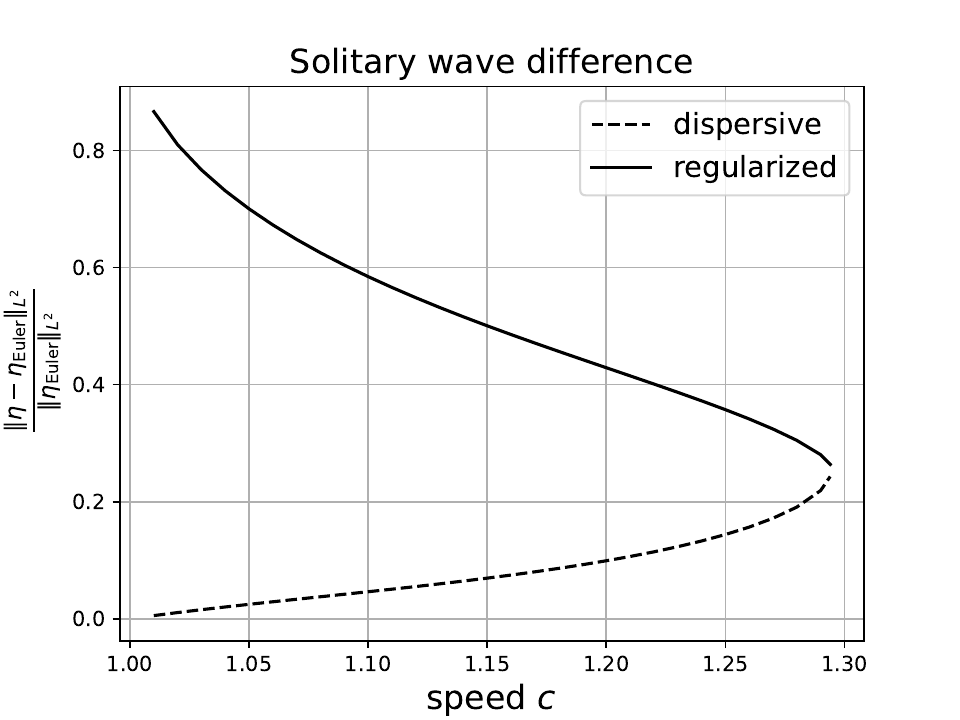}
	}
	\caption
	{
		Accuracy of approximation of the full solitary waves.
		The solid line corresponds to the regularized
		shallow water system \eqref{regularizedSW}
		and the dashed line -- to the fully dispersive model
		\eqref{regularizedHP}.
	}
\label{Difference}
\end{figure}
\begin{figure}[t!]
	\centering
%	\subfigure
	{
		\includegraphics[width=0.9\textwidth]
		{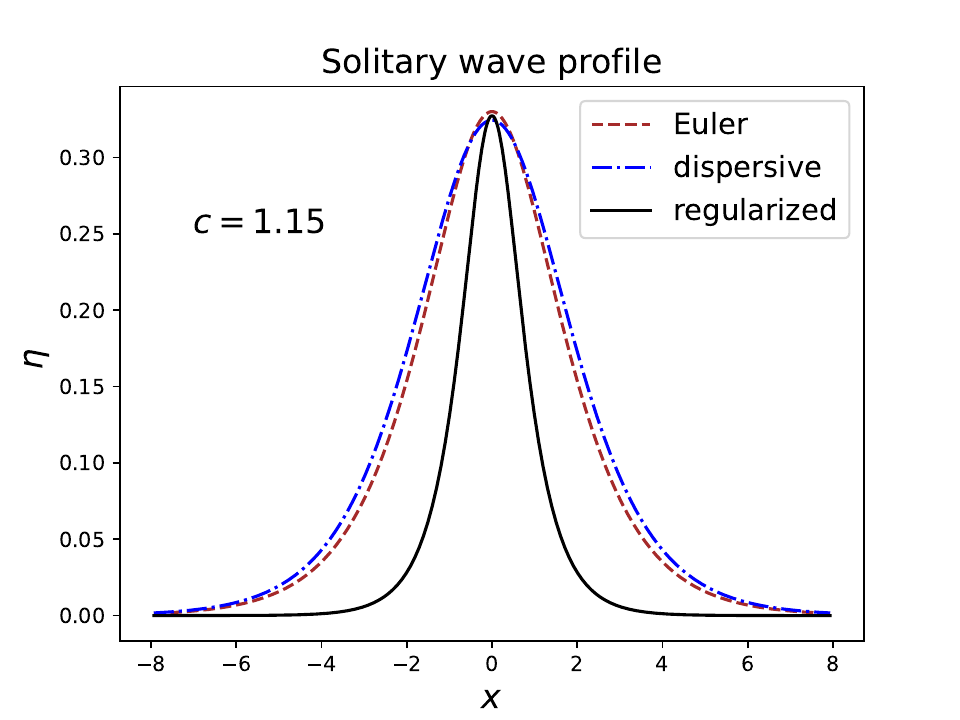}
	}
	\caption
	{
		Solitary waves corresponding to
		the Froude number $c = 1.15$.		
		The black solid line corresponds to the regularized
		shallow water system \eqref{regularizedSW}
		and the blue dash-dot line -- to the fully dispersive model
		\eqref{regularizedHP},
		the brown dashed line -- to the Euler system.
	}
\label{Soliton}
\end{figure}

A fundamental question for nonlinear dispersive wave models
is whether they admit solitary-wave solutions.
In the case of the regularized system \eqref{regularizedSW},
we address this question at a numerical level
and find clear evidence for the existence of solitary waves.
Establishing their existence rigorously would require a dedicated analytical framework
and is not pursued here. 
We look for solutions $\eta(x, t) = \eta(x - ct)$
and $u(x, t) = u(x - ct)$,
where we keep the same notation
$\eta$, $u$ for solitary wave profiles.
In our non-dimensional framework
speed $c$ coincides with the Froude number.
The regularized shallow water system \eqref{regularizedSW}
reduces to
\begin{align}
\label{solitary_sys1}
	c\eta &=
	u + K (\eta u)
	, \\
\label{solitary_sys2}
	cu &=
	\eta + K u^2 / 2
	.
\end{align}
It can be rewritten as $\mathcal L(\eta, u) = \mathcal N(\eta, u)$,
where
the linear $\mathcal L$ and the non-linear $\mathcal N$ parts
are as follows
\[
	\mathcal L(\eta, u)
	=
	\begin{pmatrix}
		c & -1
		\\
		- 1 & c
	\end{pmatrix}
	\begin{pmatrix}
		\eta
		\\
		u
	\end{pmatrix}
	, \quad
	\mathcal N(\eta, u)
	=
	\begin{pmatrix}
		K (\eta u)
		\\
		K u^2 / 2
	\end{pmatrix}
	.
\]
In this form, the system can be solved numerically
applying the Petviashvili iteration method
\cite{Clamond_Dutykh2013}.
The iterative scheme is defined via
\[
	( \eta_{n+1}, u_{n+1} )^T
	=
	S_n^2 \mathcal L^{-1}( \mathcal N( \eta_n, u_n ) )
\]
where $S_n$ is a stabilisation factor computed by
\[
	S_n = \frac
	{ \int ( \eta_n, u_n ) \mathcal L( \eta_n, u_n ) dx }
	{ \int ( \eta_n, u_n ) \mathcal N( \eta_n, u_n ) dx }
	.
\]
Note that matrix $\mathcal L$
is invertible if and only if
$c^2 > 1$.
The computational domain is set to $[-140, 140]$
with 8192 Fourier modes.
The iterative procedure
is terminated when the difference
between
\(
	( \eta_{n+1}, u_{n+1} )
\)
and
\(
	( \eta_{n}, u_{n} )
\)
in $L^2$-norm is less than $1.8 \cdot 10^{-13}$.

For the reference,
a similar splitting is applied to
the fully dispersive system \eqref{regularizedHP} in \cite{Dinvay}.
Additionally, we compute the Euler solitary waves
that can be transformed by means of a conformal-mapping to solutions of Babenko's equation.
For details we refer to \cite{Dutykh_Clamond2014},
because it gives the mathematical essence of the conformal-mapping,
Babenko-type and Petviashvili framework.
The concise original presentation of the algorithmic approach used
for the steady Euler solitary-wave calculations can be found in \cite{Clamond_Dutykh2013}.
It is implemented in the freely available
software \cite{Dutykh_code}.
Furthermore,
\cite{Clamond_Dutykh_Duran2015}
extends the same computational philosophy
to generalized solitary and capillary-gravity waves and therefore places the present
numerical solitary-wave study in a broader, well-established computational context.
The paper \cite{Dutykh_Clamond_Duran2016} provides a more detailed computational
treatment of generalized solitary waves and broadens the methodological basis of the code
and iterative framework.
Finally, \cite{Clamond_Dutykh2018} presents a mature and comprehensive
arbitrary-depth formulation of the steady-wave computation framework.

In Figure \ref{AmplitudeSpeed}
one can see the dependence of amplitude $a = \eta(0)$
of solitary waves on their speed $c$.
The brown dashed line corresponds to the full Euler model,
with the highest wave having the speed $c \approx 1.29421$.
We also report a few values in the Table \ref{amplitude_speed_table}.
\begin{table}[h!]
    \centering
    \renewcommand{\arraystretch}{1.4} 
    \begin{tabular}{|c|c|c|c|}
        \hline
         $c$        & 1.1 & 1.2 & 1.3
         \\ 
        \hline
         $\eta(0)$
         &
         0.2153082668048
         &
         0.4425522140106
         &
         0.68201597861631
         \\ 
        \hline
         $u(0)$
         &
         0.2082890214947
         &
         0.414330282965
         &
         0.61820536441872
         \\ 
        \hline
    \end{tabular}
    \caption{
        Solitary wave amplitudes for the regularized shallow water system \eqref{regularizedSW}.
    }
    \label{amplitude_speed_table}
\end{table}
The relations associated with models
\eqref{regularizedSW} and \eqref{regularizedHP}
are very close.
However, as one may anticipate, the fully dispersive
model \eqref{regularizedHP}
approximates Euler solitary waves better
than \eqref{regularizedSW}.
In Figure \ref{Difference} one can see the dependence on speed of
the relative difference
%$\epsilon(c) = d( \eta_0, \eta)$,
\(
	\epsilon(c) =
	{ \lVert \eta_{\text{Euler}} - \eta \rVert _{L^2} }
	/	
	{ \lVert \eta_{\text{Euler}} \rVert _{L^2} }
	,
\)
where the Euler solitary wave $\eta_{\text{Euler}}$
and $\eta$, standing either for \eqref{regularizedSW}
or \eqref{regularizedHP}, correspond
to the same speed $c$.
It turns out that the shallow water regime
results in narrower solitary waves, but of the same size
(see Figure \ref{Soliton}).
Nevertheless,
for high and fast waves, the regularized shallow water system \eqref{regularizedSW} exhibits accuracy comparable to that of the fully dispersive system \eqref{regularizedHP},
as evident in Figure \ref{Difference}.

%%%%%%%%%%%%%%%%%%%%%%%%%%%%%%%%%%%%%%%%%%%%%%%%%%%%%%%%%%%%%%%%%%%%%%%%%%%%%
\section{Conclusions}
%%%%%%%%%%%%%%%%%%%%%%%%%%%%%%%%%%%%%%%%%%%%%%%%%%%%%%%%%%%%%%%%%%%%%%%%%%%%%
\setcounter{equation}{0}
%%%%%%%%%%%%%%%%%%%%%%%%%%%%%%%%%%%%%%%%%%%%%%%%%%%%%%%%%%%%%%%%%%%%%%%%%%%%%
%
\noindent
We introduced a shallow water model \eqref{regularizedSW} that has the potential
to substitute the standard system \eqref{SWsystem} in practice.
The analytical sections provide rigorous local and small-data global well-posedness results.
This is an improvement on the shallow-water system which can feature shock formation for
data of any size which prevents global well-posedness of smooth solutions \cite{Evans}. Even popular regularizations
of the shallow-water system may feature large-time, but not necessarily global existence \cite{Israwi2011}.
Our numerical simulations support
the claim that the proposed model can propagate steep profiles without the same visible shock
formation displayed by the standard shallow water system in the reported experiment.
However,
a broader claim of complete shock suppression for all data
can only be formulated as a hypothesis at the current stage,
since it can only be based on an arbitrary-data global dynamics analysis.
We computed numerically solitary waves for \eqref{regularizedSW}.
A rigorous analytical proof of their existence will hopefully be presented elsewhere.

%
%---------------------------------------------------------------------
%
%
\vskip 0.05in
\noindent
{\bf Acknowledgments.}
{
    ED is supported by the Research Council of Norway through
    its Centres of Excellence scheme (Hylleraas centre, 262695).
}
%
%
%----------------------------------------------------------------------------

%%%%%%%%%%%%%%%%%%%%%%%%%%%%%%%%%%%%%%%%%%%%%%%%%%%%%%%%%%%%%%%%%%%%%%%%%%%%%
%%%%%%%%%%%%%%%%%%%%%%%%%%%%%%%%%%%%%%%%%%%%%%%%%%%%%%%%%%%%%%%%%%%%%%%%%%%%%
%%%%%%%%%%%%%%%%%%%%%%%%%%%%%%%%%%%%%%%%%%%%%%%%%%%%%%%%%%%%%%%%%%%%%%%%%%%%%
\bibliographystyle{abbrvnat}
\bibliography{bibliography}

\begin{thebibliography}{99}
\small
\renewcommand{\baselinestretch}{0.95}
\setlength{\itemsep}{-0.5mm}


\bibitem{Aceves_Sanchez_Minzoni_Panayotaros}
Aceves-S\'{a}nchez, P., Minzoni, A.A. and Panayotaros, P. 
{\em Numerical study  of a nonlocal model for water-waves with variable depth}, 
Wave Motion {\bf 50} (2013), 80--93.


\bibitem{Ali_Kalisch2010} A. Ali and H. Kalisch,
        {\it Energy balance for undular bores},
          Compt. Rend. Mecanique, {\bf 338} (2010), 67-70. 


\bibitem{Bondehagen_Roeber_Kalisch_Buckley2024}
A. Bondehagen, V. Roeber, H. Kalisch, M. Buckley, M. Stre{\ss}er, M. Cysewski, J. Horstmann, 
M. Bj{\o}rnestad, O.E. Ige, H.G. Fr{\o}ysa and R. Carrasco-Alvarez,
{\em Wave-driven current and vortex patterns at an open beach: 
Insights from phase-resolving numerical computations and {L}agrangian measurements},
Coastal Engineering {\bf 193} (2024) p.104591.



\bibitem{Brezis_Gallouet}
H. Brezis, T. Gallouet.
Nonlinear Schrödinger evolution equations.
Nonlinear Analysis: Theory, Methods \& Applications;
Volume 4, Issue 4, 1980, Pages 677-681
https://doi.org/10.1016/0362-546X(80)90068-1


\bibitem{Brezis_Wainger}
H. Brezis, S. Wainger.
A note on limiting cases of sobolev embeddings
and convolution inequalities.
Communications in Partial Differential Equations,
Volume 5, 1980 - Issue 7
https://doi.org/10.1080/03605308008820154


\bibitem{Clamond_Dutykh2013}
D. Clamond, D. Dutykh,
{\em Fast accurate computation of the fully nonlinear
solitary surface gravity waves},
Comput. \& Fluids 84 (2013) 35–38.


\bibitem{Dutykh_code}
D. Clamond, D. Dutykh, 2012.
http://www.mathworks.com/matlabcentral/fileexchange/
39189-solitary-water-wave.

\bibitem{Courant_Friedrichs1999} R. Courant and K.O. Friedrichs
Supersonic flow and shock waves (Vol. 21). Springer Science \& Business Media, 1999.

\bibitem{Dinvay}
E. Dinvay,
{\em On well-posedness of a dispersive system of
the Whitham--Boussinesq type},
Applied Mathematics Letters,
Volume 88, February 2019, Pages 13-20.
%https://doi.org/10.1016/j.aml.2018.08.005




\bibitem{Dinvay_Dutykh_Kalisch}
Dinvay, E., Dutykh, D., Kalisch, H.
{\em A comparative study of bi-directional Whitham systems}.
Applied Numerical Mathematics.


\bibitem{Dinvay_Tesfahun} E. Dinvay, S. Selberg and A. Tesfahun,
{\em Well-Posedness for a Dispersive System of the Whitham--Boussinesq Type}, 
SIAM Journal on Mathematical Analysis {\bf 52} (2020), 2353--2382.
%(3), pp.2353-2382.





\bibitem{Evans} L.C. Evans, Partial differential equations.  First edition. 
Graduate Studies in Mathematics, 19
(American Mathematical Society, Providence, RI, 1998).



\bibitem{Favre} H. Favre, 
       {\em Ondes de Translation},
        (Dunod, Paris, 1935).


\bibitem{Israwi2011}
S. Israwi, 
{\em Large time existence for 1D Green-Naghdi equations}, 
Nonlinear Analysis {\bf bf 74} (2011), 81--93.

\bibitem{Kalisch_Lagona_Roeber2024} H. Kalisch, F. Lagona and V. Roeber,
{\em Sudden wave flooding on steep rock shores: a clear but hidden danger}, Natural Hazards {\bf 120} (2024), 3105--3125.
%(3), pp.3105-3125.

\bibitem{Lannes} D. Lannes, The Water Wave Problem.
                 Mathematical Surveys and Monographs, vol. {\bf 188}
                (Amer. Math. Soc., Providence, 2013).


\bibitem{LeMehaute} B. LeMehaute,
An introduction to hydrodynamics and water waves,
(Springer Verlag, Berlin, 1976).


\bibitem{LeVeque} R.J. LeVeque, 
Finite volume methods for hyperbolic problems. Cambridge Texts in Applied Mathematics
(Cambridge University Press, Cambridge, 2002).

\bibitem{Moldabayev_Dutykh}
Moldabayev, D., Kalisch, H. and Dutykh, D. 
{\em The Whitham Equation as a model for surface water waves},
Phys. D  {\bf 309} (2015), 99--107.

\bibitem{Roeber2010} V. Roeber, K.F. Cheung and M.H. Kobayashi,
{\em Shock-capturing Boussinesq-type model for nearshore wave processes},
Coastal Engineering {\bf 57} (2010), 407--433.


\bibitem{Stoker} J.~J. Stoker,
         {\em Water Waves: The Mathematical Theory with Applications}, 
          Pure and Applied Mathematics, Vol. IV. 
          (Interscience Publishers, New York, 1957).




\bibitem{Toro} E.F. Toro, 
Riemann solvers and numerical methods for fluid dynamics. A practical introduction. Third edition
(Springer Verlag, Berlin, 2009).



\bibitem{Ursell1953} F. Ursell
{\em The long-wave paradox in the theory of gravity waves},
Math. Proc. Cambridge Phil. Soc. {\bf 49} (1953), 685--694.
%Issue 4


\bibitem{Whitham} G. B. Whitham, 
       {\em Linear and Nonlinear Waves},
        (Wiley, New York, 1974).



\end{thebibliography}
%%%%%%%%%%%%%%%%%%%%%%%%%%%%%%%%%%%%%%%%%%%%%%%%%%%%%%%%%%%%%%%%%%%%%%%%%%%%%
%%%%%%%%%%%%%%%%%%%%%%%%%%%%%%%%%%%%%%%%%%%%%%%%%%%%%%%%%%%%%%%%%%%%%%%%%%%%%
%%%%%%%%%%%%%%%%%%%%%%%%%%%%%%%%%%%%%%%%%%%%%%%%%%%%%%%%%%%%%%%%%%%%%%%%%%%%%
\begin{comment}

\end{comment}

\end{document}